\documentclass[10pt]{amsart}
     \makeatletter
     \def\section{\@startsection{section}{1}%
     \z@{.7\linespacing\@plus\linespacing}{.5\linespacing}%
     {\bfseries
     \centering
     }}
     \def\@secnumfont{\bfseries}
     \makeatother
\newtheorem{theorem}{Theorem}[section]
\newtheorem{lemma}[theorem]{Lemma}

\newtheorem{corollary}[theorem]{Corollary}
\theoremstyle{definition}
\newtheorem{definition}[theorem]{Definition}
\newtheorem{example}[theorem]{Example}
\theoremstyle{remark}

\numberwithin{equation}{section}
\usepackage{amsmath,amsthm,amssymb,amsbsy}

\usepackage[all]{xy}
\usepackage{chngcntr}
\usepackage{graphicx}

\counterwithin{figure}{section}

\newcommand{\MM}{\mathbb{M}}
\newcommand{\NN}{\mathbb{N}}

\newcommand{\RR}{\mathbb{R}}

\newcommand{\EEE}{\mathcal{E}}
\newcommand{\FFF}{\mathcal{F}}

\newcommand{\LLL}{\mathcal{L}}

\newcommand{\df}[1]{\,\mathrm{d}#1}                         

\newcommand{\eps}{\varepsilon}                              

\begin{document}

\setlength{\parindent}{0cm}
\setlength{\parskip}{0.5cm}

\title[Changes of measure for counting processes]{Exponential martingales and changes of measure for counting processes}

\author[A. Sokol and N. R. Hansen]{Alexander Sokol and Niels Richard Hansen}

\address{Alexander Sokol: Institute of Mathematics, University of
  Copenhagen, 2100 Copenhagen, Denmark, alexander@math.ku.dk}

\address{Niels Richard Hansen: Institute of Mathematics, University of
  Copenhagen, 2100 Copenhagen, Denmark,niels.r.hansen@math.ku.dk}



\begin{abstract}
We give sufficient criteria for the Dol{\'e}ans-Dade exponential of a
stochastic integral with respect to a counting process local martingale to be a
true martingale. The criteria are adapted particularly to the case of counting
processes and are sufficiently weak to be useful and verifiable, as we illustrate by several
examples. In particular, the criteria allow for the construction of for example
nonexplosive Hawkes processes as well as counting processes with stochastic intensities
depending on diffusion processes.
\end{abstract}

\maketitle

\setlength{\parindent}{0cm}
\setlength{\parskip}{0.5cm}

\section{Introduction}

Statistical counting processes models are fundamental for the modeling of
events occurring in continuous time. The simplest such process is the Poisson
process, where the intensity of new events is constant. For more complex phenomena,
it is necessary to consider models based on counting process with stochastic intensities, see
e.g. \cite{MR636252}. Such
models find applications in subjects as diverse as observational studies, neuronal
spike trains, molecular biology and corporate defaults, see e.g.
\cite{MR2817610,Truccolo2005,Masud2011,CarstensenEtAl2010,AzizEtAl2012}. Therefore, as discussed by Gjessing et al.
in \cite{MR2726223}, for the purpose of statistical modeling
using counting processes, it is of interest to formulate generic statistical models
based on counting processes with stochastic intensity in terms of a family of candidate
intensities, and it is then essential to be able to verify that the intensities result in well-defined
nonexplosive models.

The purpose of this paper is to obtain results for constructing
nonexplosive counting processes with stochastic intensities via a change of measure on the
background probability space. The objective is to derive verifiable
conditions in a counting process context for an appropriate exponential martingale
to be a true martingale. Using a change of measure, this allows for construction of nonexplosive counting processes. 

To this end, we need conditions on the candidate
intensities. If the intensity is adapted to the filtration generated
by the counting process itself, precise results are obtainable by
transferring the problem to a canonical setup, see e.g. Theorem 5.2.1 in
\cite{MR2189574}. If $N$ is a counting process, which, under some probability measure $P$, is a
homogeneous Poisson counting process, this theorem shows that given a candidate
intensity process $\lambda$ satisfying $\lambda_t \leq a(N_{t-})$ for a sequence $a(n)$ such
that $\sum_{n=1}^\infty 1/ a(n)$ diverges, then there is a measure $Q$ with likelihood process with respect to $P$ being an
exponential martingale, such that $N$ is a nonexplosive counting
process with intensity $\lambda$ under $Q$. This result is mentioned
in \cite{MR2726223} as the Jacobsen condition. It holds on the canonical
spaces considered in \cite{MR2189574} and allows for
the construction of a counting process on bounded intervals with
intensity $\lambda$ by a change of measure.

Alternative approaches to ensure the existence of a nonexplosive
counting process with a given intensity are surveyed in \cite{MR2726223}, but 
a general yet verifiable condition is difficult to obtain. The most general, explicit
condition mentioned in \cite{MR2726223} is (25). This is a growth condition on
$\lambda_t^{a}$ with $a > 1$, which is typically too strong
or difficult to verify in practice. A consequence of our results is
that a growth condition with $a = 1$ is sufficient, which
is much more useful.

The starting point for our sufficient criteria is the paper by L{\'e}pingle and M{\'e}min, \cite{MR489492}, and
their general results, which we adapt to the specific case of 
Dol{\'e}ans-Dade exponentials of stochastic integrals with respect to counting
process local martingales. Decomposing the appropriate exponential martingale
into two factors, we apply a predictable criterion (based on a moment condition
for a predictable process) and an optional criterion (based on a moment condition for
an optional process) appropriately for each part. For the particular case where the initial
counting process is a homogeneous Poisson process, the part of the moment condition initiated
by the predictable criterion ultimately vanishes, yielding a useful criterion for the exponential martingale
to be a true martingale which has no direct analogue in the literature for general martingales.

As we demonstrate by example, our criterion suffices to prove nonexplosion of counting
processes for a variety of stochastic intensity processes. In particular, due to the
technique based on changes of measure, our criterion is applicable not only for intensities
depending on the counting process itself, but also for intensities depending on auxiliary
processes such as diffusions, allowing for the construction of models with interacting
diffusion and jump processes.

\section{Summary of results}
\label{sec:Summary}

In this section, we state and discuss our main results, postponing proofs to Section \ref{sec:SuffUI}. Consider a filtered probability space
$(\Omega,\FFF,(\FFF_t)_{t\ge0},P)$ satisfying the usual conditions, see \cite{MR2273672}, Section I.1 for the definition of
this as well as other standard probabilistic concepts. We say that $N$ is a
nonexplosive $d$-dimensional counting process if $N$ is c\`{a}dl\`{a}g and piecewise
constant with jumps of size one, and no coordinates of $N$ jump at the
same time. We say that a process $X$ is locally bounded if there is a sequence of
stopping times increasing almost surely to infinity such that $X^{T_n}1_{(T_n>0)}$ is bounded. Let $\lambda$ be a nonnegative,
predictable and locally bounded $d$-dimensional process. Then
$\lambda$ is almost surely integrable on compacts with respect to
the Lebesgue measure. We say that $N$ is a counting process with intensity $\lambda$ if it holds that $N_t^i - \int_0^t \lambda^i_s\df{s}$
is a local martingale for each $i$. Note in particular that
since the predictable $\sigma$-algebra considered is the one generated
by the filtration $(\FFF_t)_{t\ge0}$, the intensity is allowed to
depend on other processes than just $N$.

We recall the definition of
Dol{\'e}ans-Dade exponentials. In the following, all semimartingales $X$ are assumed to have c\`{a}dl\`{a}g paths,
that is, $X(\omega)$ is c\`{a}dl\`{a}g for all $\omega\in\Omega$. By
$X_{t-}$, we denote the limit of $X_s$ as $s$ tends to $t$ from below,
and we write $\Delta X_t = X_t-X_{t-}$ for the jump of $X$ at $t$. Assume given a semimartingale $X$ with initial value
zero. By Theorem II.37 of \cite{MR2273672} and Theorem I.4.61 of
\cite{MR1943877}, the stochastic differential equation $Z_t = 1 + \int_0^t
Z_{s-}\df{X}_s$ has a c\`{a}dl\`{a}g adapted solution, unique up to
indistinguishability, and the solution is
\begin{align}
  \label{eq:ExpSMGDef}
  \EEE(X)_t &= \exp\left(X_t-\frac{1}{2}[X^c]_t\right)\prod_{0<s\le t}(1+\Delta X_s)\exp(-\Delta X_s),
\end{align}
where $X^c$ is the continuous
martingale part of $X$, see Proposition I.4.27 of \cite{MR1943877}, and
$[X^c]$ denotes the quadratic variation process. If $X$ is a local martingale, $\EEE(X)$ is a local martingale as
well, and in this case, we refer to $\EEE(X)$ as an exponential
martingale. The case where $\Delta X\ge-1$ will be of particular
importance to us. In this case, $\EEE(X)$ is nonnegative, and we may put $R=\inf\{t\ge0\mid \Delta X_t=-1\}$ and obtain
\begin{align}
  \label{eq:ExpSMGDefStopped}
  \EEE(X)_t &= 1_{(t<R)}\exp\left(X_t-\frac{1}{2}[X^c]_t+\sum_{0<s\le t}\log(1+\Delta X_s)-\Delta X_s\right).
\end{align}

Now assume given a $d$-dimensional nonexplosive counting process $N$ with nonnegative,
predictable and locally bounded intensity $\lambda$, and assume given another
$d$-dimensional nonnegative, predictable and locally bounded process
$\mu$.

\begin{definition}
\label{def:Compatibility}
 We say that $\mu$ is $\lambda$-compatible if it holds for all
 $\omega\in\Omega$ that
 $\mu^i_t(\omega)=0$ whenever $\lambda^i_t(\omega)=0$, and if the process $\gamma$ defined by
$\gamma^i_t=\mu^i_t(\lambda^i_t)^{-1}$ for $i\le d$ is locally bounded.
\end{definition}

In Definition \ref{def:Compatibility}, we use the convention that zero
divided by zero is equal to one. Now assume that $\mu$ is $\lambda$-compatible. Define $M$ to be
the $d$-dimensional local martingale given by $M^i_t = N^i_t-\int_0^t\lambda^i_s\df{s}$. Put $\gamma^i_t=\mu^i_t(\lambda^i_t)^{-1}$ and
$H^i_t=\gamma^i_t-1$ for $t\ge0$. As we have assumed that $\mu$ is
$\lambda$-compatible, $\gamma$ and $H$ are both well-defined and locally
bounded real-valued processes. We define $(H\cdot
M)_t=\sum_{i=1}^d \int_0^t H^i_s\df{M^i}_s$, $H\cdot M$ is then a
one-dimensional process.

The following lemma shows that given $\lambda$ and $\mu$, $\EEE(H\cdot M)$ is the
relevant exponential martingale to consider for changing the
distribution of $N$ from a counting process with intensity $\lambda$ to a counting process with intensity
$\mu$, where $H^i_t=\mu^i_t(\lambda^i)^{-1}_t-1$. 

\begin{lemma}
\label{lemma:PointProcessMeasChg}
Let $T$ be a stopping time and assume that $\EEE(H\cdot M)^T$ is a uniformly integrable
martingale. With $Q$ being the probability measure with Radon-Nikodym derivative $\EEE(H\cdot M)_T$ with respect
to $P$, it holds that $N$ is a counting process under $Q$ with
intensity $1_{[0,T]}\mu+1_{(T,\infty)}\lambda$. In particular, if
$\EEE(H\cdot M)$ is a martingale, it holds for any $t\ge0$ and with
$Q_t$ being the probability measure with Radon-Nikodym derivative $\EEE(H\cdot M)_t$ with respect
to $P$ that $N$ is a counting process under $Q_t$ with
intensity $1_{[0,t]}\mu+1_{(t,\infty)}\lambda$.
\end{lemma}

In general, we cannot expect $\EEE(H\cdot M)$ to be a uniformly integrable martingale, only an ordinary
martingale, because the distributions of counting processes with
intensities which differ sufficiently will in general be singular. For example,
the distributions of two homogeneous Poisson processes with different
intensities are singular, see Proposition 3.24 of \cite{MR1113698}.

As an aside, note that the measure $Q_T$ obtained in Lemma
\ref{lemma:PointProcessMeasChg} of course always will be absolutely
continuous with respect to $P$. A natural question to ask is when $Q_T$
and $P$ will be equivalent. This is the case when the Radon-Nikodym
derivative is almost surely positive. Lemma \ref{lemma:QPEquivalence} gives a
condition for this to be the case.

\begin{lemma}
\label{lemma:QPEquivalence}
If the set of zeroes of $\mu$ has Lebesgue measure zero, $\EEE(H\cdot
M)$ is almost surely positive.
\end{lemma}

Finally, we state our sufficient criteria for $\EEE(H\cdot M)$ to be a
true martingale. Defining $\log_+ x=\max\{0,\log
x\}$ for $x\ge0$, with the convention that the logarithm of zero is
minus infinity, our main results are the following.

\begin{theorem}
\label{theorem:mainPointProcessMGCrit}
Assume that $\lambda$ and $\mu$ are nonnegative, predictable and
locally bounded. Assume that $\mu$ is $\lambda$-compatible. It holds that $\EEE(H\cdot M)$ is a martingale if there is an $\eps>0$ such that
whenever $0\le u\le t$ with $t-u \le\eps$, one of the following two
conditions are satisfied:
\begin{align}
  &E\exp\left(\sum_{i=1}^d\int_u^t (\gamma^i_s\log\gamma^i_s-(\gamma^i_s-1))\lambda^i_s\df{s}\right)<\infty\quad\textrm{or}\label{eq:MainCrit1}\\
  &E\exp\left(\sum_{i=1}^d \int_u^t\lambda_s^i\df{s}+\int_u^t\log_+ \gamma^i_s\df{N^i}_s\right)<\infty.\label{eq:MainCrit2}
\end{align}
\end{theorem}

\begin{corollary}
\label{coro:mainPoissonMGCrit}
Assume that $\lambda=1$ and assume that $\mu$ is nonnegative,
predictable and locally bounded. Then $\mu$ is $\lambda$-compatible. It holds that $\EEE(H\cdot M)$ is a martingale if there is an $\eps>0$ such
that whenever $0\le u\le t$ with $t-u\le\eps$, one of the following
two conditions are satisfied:
\begin{align}
  &E\exp\left(\sum_{i=1}^d\int_u^t \mu^i_s\log_+\mu^i_s\df{s}\right)<\infty \quad\textrm{ or }\label{eq:CoroCrit1}\\
  &E\exp\left(\sum_{i=1}^d\int_u^t\log_+ \mu^i_s\df{N^i}_s\right)<\infty.\label{eq:CoroCrit2}
\end{align}
\end{corollary}

The immediate use of Theorem \ref{theorem:mainPointProcessMGCrit} and its
corollary is as an existence result for nonexplosive counting processes with
particular intensities, as the change of measure obtained from the martingale property of
$\EEE(H\cdot M)$ yields the existence of a nonexplosive counting process
distribution with given intensity $\mu$ on a bounded time interval
$[0,t]$. That this is the case may be seen from Lemma
\ref{lemma:PointProcessMeasChg}, which shows that under the measure
$Q_t$ with Radon-Nikodym derivative $\EEE(H\cdot M)_t$ with respect to
$P$, $N$ is a counting process with intensity
$1_{[0,t]}\mu+1_{(t,\infty)}\lambda$.

Note that it is not necessarily
possible to use the family $(Q_t)$ to obtain the existence of a
limiting probability measure $Q_\infty$ under which $N$ has intensity
$\mu$ on all of $\RR_+$. Such a limiting probability would require extension results
as discussed in the appendix of \cite{MR0309184}. See also the
discussion following Example \ref{example:ExistenceLinearChange}.

As a specific application of our results, let us assume that we
are interested in constructing a statistical model for a nonexplosive counting
processes. We assume given a filtered probability space
$(\Omega,\FFF,(\FFF_t)_{t\ge0},P)$ and a $d$-dimensional counting process
$N$ such that under $P$, $N^i$ has intensity $\lambda^i_t = 1$. Fix a
timepoint $t$ and let us assume that we are interested in
considering a statistical model on the time interval $[0,t]$ based
on a family of intensities $(\mu_\theta)_{\theta\in\Theta}$. If
$\mu_\theta$ satisfies the criteria of Corollary \ref{coro:mainPoissonMGCrit}, we find that
$\EEE(H_\theta\cdot M)$ is a martingale, and so $\EEE(H_\theta\cdot M)_t$ has
unit mean. Letting $Q_\theta$ be the probability measure with Radon-Nikodym
derivative $\EEE(H_\theta\cdot M)_t$ with respect to $P$, it holds
that under $Q_\theta$, $N$ is a counting process with intensity, and the
intensity is $\mu_\theta$ on $[0,t]$. Furthermore, the family
$(Q_\theta)_{\theta\in\Theta}$ is dominated by $P$, and the
likelihood function is known in explicit form. Thus, Corollary
\ref{coro:mainPoissonMGCrit} has allowed us to construct the
statistical model and prove that explosion does not occur.

As regards checking the criteria in practice, an important property to
note is that the criteria only need to be checked locally, in the
sense that it is only necessary to find some $\eps>0$ such that the
criteria holds for $0\le u\le t$ with $t-u\le\eps$. This makes
it possible to apply the criteria in several
interesting situations. In particular, it allows us to extend the
criterion $(25)$ of \cite{MR2726223} from $a>1$ to $a\ge 1$, see Example \ref{example:ExistenceLinearChange}.

Instead of considering Theorem \ref{theorem:mainPointProcessMGCrit} as
a criterion for nonexplosion, we may also think of it simply as a sufficient criterion
for the Dol{\'e}ans-Dade exponential $\EEE(M)$ of a particular type of
local martingale $M$ to be a true martingale.  For $M$ a local
martingale with $\Delta M\ge-1$ and initial value zero, the question
of when $\EEE(M)$ is a uniformly integrable martingale or a true
martingale has been treated many times in the literature, see for
example \cite{MR0312567,MR699931,MR1299529,CS2001}
for results in the case of continuous $M$, and \cite{MR489492,MR547642,MR1932378,AS2013,KlebanerLipster2014} for results in the general case.

In particular, a considerable family of criteria related to this problem
is discussed in \cite{MR1932378}. We remark that the proof of Theorem
\ref{theorem:mainPointProcessMGCrit} applies Theorem III.1 and Theorem
III.7 of \cite{MR489492}. In the parlance of \cite{MR1932378}, Theorem III.1
of \cite{MR489492} corresponds to condition $I(0,1)$. There exists a
slight improvement of condition $I(0,1)$, namely condition $I(0,1-)$,
also proven in \cite{MR1932378}. Applying this condition instead of
condition $I(0,1)$ would not lead to significant improvements of our
results. We further remark that Theorem III.7 of \cite{MR489492} does not have an analogue in
the hierarchy of \cite{MR1932378}. In general, the criteria on which
the results of Theorem \ref{theorem:mainPointProcessMGCrit} are built
are among the strongest known, and optimality properties of these
criteria are known, see \cite{MR489492}. Therefore, we expect that no significant
improvements of Theorem \ref{theorem:mainPointProcessMGCrit} will be
possible. Furthermore, we remark that while the conditions (\ref{eq:MainCrit1}) and
(\ref{eq:CoroCrit1}) are rather straightforward localised versions of results
of \cite{MR489492}, the conditions (\ref{eq:MainCrit2}) and particularly
the simplified variant (\ref{eq:CoroCrit2}) do not have direct analogoues in
the hierarchy of \cite{MR1932378}.


The remainder of the paper is organized as follows. In Section
\ref{sec:Examples}, we give some examples of applications of the
results. In Section \ref{sec:SuffUI}, we present the proof of the main
results. Appendix \ref{sec:Proofs} contains supplementary results for
Section \ref{sec:Examples}.

\section{Examples}
\label{sec:Examples}

In this section, we give examples where the conditions in Theorem
\ref{theorem:mainPointProcessMGCrit} and Corollary
\ref{coro:mainPoissonMGCrit} may be verified. Our first example shows how Theorem
\ref{theorem:mainPointProcessMGCrit} under certain circumstances
allows for changes of the intensity where
the new intensity is an affine function of the old intensity. Such
criteria were also discussed in Theorem 2 of \cite{MR2817610}, where the new
intensity $\mu$ was assumed to be related to the initial intensity
$\lambda$ by the relationship $|\mu_t-\lambda_t|\le
\theta\sqrt{\lambda_t}$.

\begin{example}
\label{example:AffineInPreviousIntensity}
Assume that $d$ is equal to one. Assume that $\lambda_s\ge\delta$ for some $\delta>0$ and that $\mu_t\le\alpha+\beta\lambda_s$. If there is $\eps>0$ such that for
$0\le u\le t$ with $t-u\le\eps$, $\int_u^t\lambda_s\df{s}$ has an
exponential moment of order $(1+(\alpha\delta^{-1}+\beta)\log_+(\alpha\delta^{-1}+\beta))$, then $\EEE(H\cdot M)$ is a martingale.
\hfill$\circ$
\end{example}

\textbf{\textit{Proof of Example \ref{example:AffineInPreviousIntensity}:\quad}} By our assumptions, $\gamma_t = \alpha\lambda^{-1}_t+\beta\le
\alpha\delta^{-1}+\beta$. Using that $x\log
x-(x-1)\le 1+x\log x\le 1+x\log_+x$ for any $x\ge0$, we then obtain
\begin{align}
  \int_u^t (\gamma^i_s\log\gamma^i_s-(\gamma^i_s-1))\lambda_s\df{s}
  \le (1+(\alpha\delta^{-1}+\beta)\log_+(\alpha\delta^{-1}+\beta))\int_u^t \lambda_s\df{s},
\end{align}
so the first criterion of Theorem \ref{theorem:mainPointProcessMGCrit}
yields the result.
{\hfill $\Box$}

In the remainder of the examples, we will assume that $\lambda=1$,
such that $N$ is a $d$-dimensional standard Poisson process, and give
particular cases where Corollary \ref{coro:mainPoissonMGCrit} may be
applied. For Example \ref{example:Novikov} below, we first introduce
some notation. Let $X$ be a semimartingale. If the quadratic variation
process $[X]$ is locally integrable, the dual predictable projection $\Pi^*_p[X]$ is
well-defined, see Definition 5.21 of \cite{MR1219534} and Section III.5 of \cite{MR2273672}. In this case, we put $\langle X\rangle=\Pi^*_p[X]$ and
refer to $\langle X\rangle$ as the predictable quadratic variation
process of $X$. It then holds that $\langle X\rangle$ is predictable,
and in the case where $X$ is a locally square-integrable local
martingale, both $[X]-\langle X\rangle$ and $X^2-\langle X\rangle$ are
local martingales, see also Section I.4 of \cite{MR1943877}.

\begin{example}
\label{example:Novikov}
Assume that $\mu$ is a nonnegative, predictable and locally integrable
process, and assume that there is $\eps>0$ such that $\exp(\eps\langle
H\cdot M\rangle_t)$ is integrable for all $t\ge0$. In this case, the first
criterion of Corollary \ref{coro:mainPoissonMGCrit} may be applied to
show that $\EEE(H\cdot M)$ is a martingale.
\hfill$\circ$
\end{example}

\textbf{\textit{Proof of Example \ref{example:Novikov}:\quad}} Let $\eps>0$
be given such that $\exp(\eps\langle H\cdot M\rangle_t)$ is integrable
for all $t\ge0$. Pick $K>0$ so large that $x\log_+x\le \eps (x-1)^2$
holds for
$x\ge K$, then $E\exp(\sum_{i=1}^d\int_u^t
\mu^i_s\log_+\mu^i_s\df{s})\le\exp(dtC)E\exp(\eps\sum_{i=1}^d \int_0^t
(H^i_s)^2\df{s})$, where $C=\sup_{-1\le x\le K}x\log_+x$. As $N$ has
no common jumps, however, we have $[H\cdot M]_t=\sum_{i=1}^d\int_0^t
(H^i)_s^2\df{N^i}_s$. Therefore, as $H$ is predictable, we obtain
\begin{align}
  \langle H\cdot M\rangle_t 
  &=\Pi^*_p \sum_{i=1}^d\int_0^t (H^i_s)^2\df{N^i}_s
   = \sum_{i=1}^d \int_0^t (H^i_s)^2\df{\Pi^*_p N^i}_s
   = \sum_{i=1}^d \int_0^t (H^i_s)^2\df{s}.
\end{align}
All in all, we conclude
\begin{align}
    E\exp\left(\sum_{i=1}^d\int_u^t \mu^i_s\log_+\mu^i_s\df{s}\right)
  &\le \exp(dtC)E\exp\left(\eps\langle H\cdot M\rangle_t\right)<\infty,
\end{align}
and the result follows by Corollary \ref{coro:mainPoissonMGCrit}.
{\hfill $\Box$}

Example \ref{example:Novikov} is noteworthy because of the
following. In \cite{MR2574236}, applying the results of \cite{MR489492}, the following
Novikov-type criterion is demonstrated: If $M$ is a locally square
integrable local martingale with $\Delta M\ge-1$ and
$\exp(\frac{1}{2}\langle M^c\rangle_\infty+\langle
M^d\rangle_\infty)$ is integrable, then $\EEE(M)$ is a uniformly
integrable martingale. Here, $M^c$ and $M^d$ denote the continuous and
purely discontinuous parts of the local martingale, respectively, see
Theorem 7.25 of \cite{MR1219534}. Furthermore, in \cite{MR2574236} it is argued by example that
the constant 1 in front of $\langle M^d\rangle_\infty$ cannot in
general be exchanged with $1-\eps$ for any $\eps>0$. See also \cite{AS2013}
for more results of this type. Example \ref{example:Novikov}, however, shows that when proving the martingale
property instead of the uniformly integrable martingale property, for the particular type of
local martingale considered here, the constant 1 may in fact be
exchanged with any positive number.

\begin{example}
\label{example:ExistenceLinearChange}
Assume that $\mu$ is a nonnegative, predictable and locally integrable process satisfying $\mu^i_t\le\alpha+\beta\sum_{j=1}^d N^j_{t-}$. Then both criteria
of Corollary \ref{coro:mainPoissonMGCrit} may be applied to obtain
that $\EEE(H\cdot M)$ is a martingale.
\hfill$\circ$
\end{example}

\textbf{\textit{Proof of Example \ref{example:ExistenceLinearChange}:\quad}}
Consider first using the first moment condition of Corollary
\ref{coro:mainPoissonMGCrit}. As $x\log_+x$ is increasing in $x$, it suffices to consider the case
where $\alpha>1$ and $\beta>0$, such that $\mu$ is positive. Fix $\eps>0$, and let $0\le u\le
t$ with $t-u\le\eps$. We then obtain, with $N^S_t = \sum_{j=1}^d N^j_t$,
\begin{align}
  \exp\left(\sum_{i=1}^d\int_u^t\mu^i_s\log_+\mu^i_s\df{s}\right)
  &\le \exp\left(\eps d (\alpha+\beta N^S_t)\log(\alpha+\beta N^S_t)\right).
\end{align}
Now, for $k$ large ehough, $\eps d(\alpha+\beta k)\log(\alpha+\beta k)\le 4\eps d\beta k\log k$. Therefore, $\exp(\sum_{i=1}^d\int_u^t\mu^i_s\log\mu^i_s\df{s})$ is
integrable if only $\exp(4\eps d \beta N^S_t\log N^S_t)$ is
integrable. $N_t^S$ is Poisson distributed with parameter $dt$, so by
choosing $\eps$ with $4\eps d\beta<1$, we obtain the desired integrability
using Lemma \ref{lemma:PoissonExpMoment}. The first moment condition
of Corollary \ref{coro:mainPoissonMGCrit} now yields the result.

If we instead wish to use the second moment condition of Corollary
\ref{coro:mainPoissonMGCrit}, define a mapping
$\varphi:\NN_0\to\RR$ by
$\varphi(n)=E\exp(\int_0^{t-u}\log\beta(n+N^S_s)\df{N^S}_s)$. Let $m\in\NN$ such that $\alpha\le \beta m$, we then obtain $\alpha+\beta x \le
\beta(m+x)$. By a conditioning argument, it then holds that
$E\exp(\sum_{i=1}^d\int_u^t\log \mu^i_s\df{N^i}_s) \le
E\varphi(m+N_u^S)$. By elementary calculations and a standard
hypergeometric summation formula, see formula (15.1.8) of
\cite{MR757537}, we obtain that whenever
$\beta(t-u)d<1$, we have
\begin{align}
  \varphi(n)=\frac{\exp(-(t-u)d)}{(1-\beta(t-u)d)^{n+1}}.
\end{align}
Further calculations then yield the bound
\begin{align}
  E\exp\left(\sum_{i=1}^d\int_u^t\log \mu^i_s\df{N^i}_s\right)
  &\le\frac{1}{(1-\beta(t-u)d)^{m+1}}\exp\left(-td+\frac{ud}{1-\beta(t-u)d}\right).
\end{align}
We conclude that the second moment condition of Corollary \ref{coro:mainPoissonMGCrit} yields the
result, using $\eps$ such that $\beta\eps d<1$.
{\hfill $\Box$}

The above is the extension of criterion
$(25)$ of \cite{MR2726223} from $a>1$ to $a\ge 1$ mentioned
earlier. For the case of intensities predictable with respect to the filtration
generated by $N$, the existence of nonexplosive counting processes with intensities affinely bounded
by the total number of jumps as in Example \ref{example:ExistenceLinearChange} is well known, see
Example 4.4.5 of \cite{MR2189574}. The abstract construction of Example
\ref{example:ExistenceLinearChange} covers the general case of
intensities predictable with respect to $(\FFF_t)$ and yields a family $(\Omega,
\FFF_t, Q_t)_{t\ge0}$ of probability spaces such that $(N_s)_{s\le t}$ has
intensity $\mu$ on $[0,t]$ under $Q_t$, here $Q_t$ is the measure
with Radon-Nikodym derivative $\EEE(H\cdot M)_t$ with respect to $P$. Additional structure on
the probability space is needed to guarantee the existence of the
inverse limit $(\Omega, \sigma(\cup_{t\ge0} \FFF_t), Q)$ with
the restriction of $Q$ to $\FFF_t$ being equal to $Q_t$, such that $(N_s)_{s\ge0}$ has
intensity $\mu$ under $Q$, see \cite{MR0096768} and \cite{MR0309184}. Such
conditions can usually be assumed fulfilled in concrete cases by
working with a suitable canonical choice of $\Omega$, but we will not
pursue this any further.

If the process $\mu$ is exactly affine in the sense
that $\mu^i_t=\alpha+\beta\sum_{j=1}^d N^j_{t-}$, and $N$ is a
homogeneous Poisson process, the martingale
property of $\EEE(H\cdot M)$ may be obtained
by direct calculation. However, this does not in itself imply that the
same result holds when we only have the inequality $\mu^i_t\le
\alpha+\beta\sum_{j=1}^d N^j_{t-}$. In general, such ``monotonicity''
properties of the martingale property for exponential martingales do
not hold, see for example \cite{MR1299529}, Example 1.13.

Next, we consider two examples involving intensities given as
solutions to stochastic differential equations. Such intensities allow
for the construction of models with interacting counting processes and
diffusions as discussed in for example \cite{AzizEtAl2012}. In both cases, we
assume given a Brownian motion relative to the given filtration
$(\FFF_t)$, meaning in the $d$-dimensional case that $(W^i)^2_t-t$ is
an $(\FFF_t)$ martingale for $i\le d$ and $W^i_tW^j_t$ is an
$(\FFF_t)$ martingale for $i,j\le d$ with $i\neq j$. We denote such a
process an $(\FFF_t)$ Brownian motion. By L{\'e}vy's
characterisation of Brownian motion for general filtered probability
spaces, see Theorem IV.33.1 of \cite{MR1780932}, this requirement ensures
that the characteristic properties of the Brownian motion interact
well with the filtration $(\FFF_t)$. By $\MM(d,d)$, we denote the set
of $d\times d$ matrices with real entries.

\begin{example}
\label{example:DiffusionPointProcessCoef}
Consider mappings $A:\NN^d_0\times \RR^d_+\to\RR^d$, $B:\NN^d_0\times \RR^d_+\to\MM(d,d)$ and
$\sigma:\NN^d_0\times \RR^d_+\to\MM(d,d)$ such that for all $\eta\in\NN_0^d$,
$A(\eta,\cdot)$, $B(\eta,\cdot)$ and $\sigma(\eta,\cdot)$ are
continuous and bounded and such that $\sigma$ always is positive definite. With $T^i_n$ denoting the
$n$'th jump time for $N^i$ and $Z^i_t = t - T^i_{N^i_t}$, let $X$
be a solution to the $d$-dimensional stochastic differential equation
\begin{align}
 \df{X}_t &= (A(N_t,Z_t)+B(N_t,Z_t)X_t)\df{t} +
\sigma(N_t,Z_t) \df{W}_t
\end{align}
with initial value $x_0$ in $\RR^d$, where $W$ is an $(\FFF_t)$ Brownian motion independent
of $N$. Let $\phi:\RR^d\to\RR_+^d$ be
Lipschitz and put $\mu_t = \phi(X_t)$. Assume that there are $\delta>0$ and
$c_A,c_B,c_\sigma>0$ such that
\begin{align}
  \sup_{t\ge0}\|A(\eta,t)\|_2 &\le c_A\|\eta\|_1^{1-\delta}\\
  \sup_{t\ge0}\|\sigma(\eta,t)\|_2 &\le c_\sigma\|\eta\|_1^{(1-\delta)/2}\\
  \sup_{t\ge0}\|B(\eta,t)\|_2 &\le c_B,
\end{align}
where $\|\cdot\|_2$ in the
first case denotes the Euclidean norm and in the two latter cases
denote the operator norm induced by the Euclidean norm, and
$\|\cdot\|_1$ denotes the $\LLL^1$ norm in $\NN_0^d$. Then, the first criterion of
Corollary \ref{coro:mainPoissonMGCrit} may be applied to obtain that
$\EEE(H\cdot M)$ is a martingale.
\hfill$\circ$
\end{example}

The interpretation of this example is as follows. The evolution of the
diffusion process depends on the counting process with mean reversion
level, mean reversion speed and diffusion coefficient which are
deterministic between jumps. The $Q$-measure, as described in Lemma
\ref{lemma:PointProcessMeasChg}, yields a model with feedback from the diffusion to the counting
process by letting the intensity be a function of the diffusion.

\textbf{\textit{Proof of Example \ref{example:DiffusionPointProcessCoef}:\quad}} We
need to show that the first criterion of Corollary
\ref{coro:mainPoissonMGCrit} is applicable. We may assume without loss
of generality that $0<\delta<1$. It suffices to prove that
for any $t>0$, $E\exp(\sum_{i=1}^d\int_0^t
\mu^i_s\log_+\mu^i_s\df{s})$ is finite. Fix $t>0$. By Jensen's inequality, we find
\begin{align}
  E\exp\left(\sum_{i=1}^d\int_0^t \mu^i_s\log_+\mu^i_s\df{s}\right)
  \le\frac{1}{t}\int_0^t E\exp\left(t\sum_{i=1}^d\mu^i_s\log_+\mu^i_s\right)\df{s}.
\end{align}
We wish to bound the expectation inside the integral by an expression
depending continuously on $s$. Recall that we have
assumed that $\phi$ is Lipschitz, so there exists
$\gamma>0$ such that $\|\phi(x)\|_\infty\le \gamma\|x\|_2$, yielding $\phi^i(x)\le
\gamma\|x\|_2$ for all $i\le d$, and so $E\exp(t\sum_{i=1}^d\mu^i_s\log_+\mu^i_s) \le E\exp(t d\gamma\|X_s\|_2\log_+\gamma\|X_s\|_2)$.
Next, let $0<\zeta<1$. It holds for all $x\ge0$ that $\log_+x\le
\zeta^{-1}x^\zeta$. Therefore, defining $\rho=t
d\gamma^{1+\zeta}\zeta^{-1}$, we conclude
\begin{align}
  E\exp\left(t\sum_{i=1}^d\mu^i_s\log_+\mu^i_s\right)
  &\le E\exp\left(\rho\|X_s\|_2^{\zeta+1}\right).
\end{align}
We will calculate this expectation by conditioning on $N$. Let $\eta$
denote a point process path, and let $(\tau_n)$ denote the event times
of $\eta$. By the explicit representation in Lemma
\ref{lemma:DiffusionPointProcessCoefExistence} as well as the results
on pathwise stochastic integration in \cite{MR1327950}, it
holds that conditionally on $N=\eta$, $X_s$ has
the same distribution as $Y^\eta_s$, where
\begin{align}
  Y^\eta_s=
      C^{-1}_s\left(x_0+\int_0^s C_vA(\eta_v,v-\tau_{\eta_v})\df{v}+\int_0^s C_v\sigma(\eta_v,v-\tau_{\eta_v})\df{W}_v\right),
\end{align}
which is a normal distribution with mean $\xi^\eta_s$ and variance
$\Sigma^\eta_s$, where
\begin{align}
  \xi^\eta_s &= C^{-1}_s\left(x_0+\int_0^s C_vA(\eta_v,v-\tau_{\eta_v})\df{v}\right)\\
  \Sigma^\eta_s &= C^{-1}_s\int_0^s (C_v\sigma(\eta_v,v-\tau_{\eta_v}))^t(C_v\sigma(\eta_v,v-\tau_{\eta_v}))\df{s}(C^{-1})_s^t,
\end{align}
and where $C_s = \exp(-\int_0^s B(\eta_v,v-\tau_{\eta_v})\df{v})$. With $\|\cdot\|_2$ denoting the matrix operator norm induced by the
Euclidean norm, Lemma \ref{lemma:NormalExpPowerMoment} yields
\begin{align}
  &E\exp(\rho \|X_s\|_2^{1+\zeta})
  =\int E\left(\left.\exp(\rho \|X_s\|_2^{1+\zeta})\right|N=\eta\right)\df{N(P)}(\eta)\\
  &\le k_dE\exp(a(\rho,\zeta)\|\xi^N_s\|^{1+\zeta})\exp\left(b(\rho,\zeta)\|\Sigma_s^N\|_2^\frac{1+\zeta}{1-\zeta}\right),\notag
\end{align}
with $a$ and $b$ as in the statement of the lemma. Next, we consider bounds for $\|\xi_s^\eta\|$ and
$\|\Sigma^\eta_s\|_2$. Note that $\|C_s\|_2 \le \exp(\int_0^s
  \|B(\eta_v,v-\tau_{\eta_v})\|_2\df{v})\le\exp(s c_B)$, where we have
  applied standard norm inequalities, see Theorem
10.10 of \cite{MR2396439} and Lemma 1.4 of \cite{MR838085}, and similarly,
$\|C_s^{-1}\|_2\le \exp(sc_B)$. Therefore, recalling that $0<\delta<1$ so that $x\mapsto x^{1-\delta}$
is increasing,
\begin{align}
  \|\xi^\eta_s\|_2
  &\le\exp(sc_B)\left(\|x_0\|_2+sc_A\exp(sc_B)\|\eta_s\|_1^{1-\delta}\right).
\end{align}
Similarly, we obtain
\begin{align}
  \|\Sigma^\eta_s\|_2&\le s\exp(4sc_B)c_\sigma^2 \|\eta_s\|_1^{1-\delta}.
\end{align}
In particular, for appropriate continuous functions $a_\xi$, $b_\xi$
and $b_\Sigma$ from $\RR_+$ to $\RR$, depending on $\zeta$, we obtain
the two bounds
\begin{align}
  \|\xi^\eta_s\|_2^{1+\zeta}
  &\le a_\xi(s)+b_\xi(s)\|\eta_s\|_1^{(1-\delta)(1+\zeta)}\\
  \|\Sigma^\eta_s\|^{\frac{1+\zeta}{1-\zeta}}_2
  &\le b_\Sigma(s)\|\eta_s\|_1^{(1-\delta)\frac{1+\zeta}{1-\zeta}}.
\end{align}
We then conclude
\begin{align}
  & E\exp(\rho \|X_s\|_2^{1+\zeta})\notag\\
  &\le k_dE\exp\left(a(\rho,\zeta)\left(a_\xi(s)+b_\xi(s)\|N_s\|_1^{(1-\delta)(1+\zeta)}\right)+b(\rho,\zeta)b_\Sigma(s)\|N_s\|_1^{(1-\delta)\frac{1+\zeta}{1-\zeta}}\right)\notag\\
  &\le k_d\exp\left(a(\rho,\zeta)a_\xi(s)\right)E\exp\left((a(\rho,\zeta)b_\xi(s)+b(\rho,\zeta)b_\Sigma(s))\|N_s\|_1^{(1-\delta)\frac{1+\zeta}{1-\zeta}}\right).
\end{align}
The above depends on given constants $\delta$, $c_A$, $c_B$ and
$c_\sigma$, as well as the constant $\zeta$ which we may choose
arbitrarily in the open interval between zero and one. We now
choose $\zeta$ so small in $(0,1)$ that
$(1-\delta)(1+\zeta)(1-\zeta)^{-1}\le 1$. Recalling that for any Poisson
distributed variable $Z$ with intensity $\lambda$ and any $c\in\RR$,
it holds that $E\exp(cZ)=\exp((\exp(c)-1)\lambda)$, we may then conclude
\begin{align}
  & E\exp(\rho \|X_s\|_2^{1+\zeta})\notag\\
  &\le k_d\exp\left(a(\rho,\zeta)a_\xi(s)\right)E\exp\left((a(\rho,\zeta)b_\xi(s)+b(\rho,\zeta)b_\Sigma(s))\|N_s\|_1\right)\notag\\
  &= k_d\exp\left(a(\rho,\zeta)a_\xi(s)\right)\exp((\exp(a(\rho,\zeta)b_\xi(s)+b(\rho,\zeta)b_\Sigma(s))-1)ds).
\end{align}
All in all, we may now define, for $0\le s\le t$,
\begin{align}
  \varphi(s)=k_d\exp\left(a(\rho,\zeta)a_\xi(s)\right)\exp((a(\rho,\zeta)\exp(b_\xi(s)+b(\rho,\zeta)b_\Sigma(s))-1)ds),
\end{align}
and obtain $E\exp(t\sum_{i=1}^d\mu^i_s\log_+\mu^i_s) \le \varphi(s)$
for all such $s$. The functions $a_\xi$, $b_\xi$ and $b_\Sigma$
depends continuously on $s$. Therefore, $\varphi$ is a continuous
function of $s$. In particular, the integral of $\varphi$ over $[0,t]$ is finite. Recalling our first
estimates, this leads us to conclude that for any $t\ge0$, it holds
that $E\exp(\sum_{i=1}^d\int_0^t \mu^i_s\log_+\mu^i_s\df{s})$ is
finite, and so the first integrability criterion of Corollary
\ref{coro:mainPoissonMGCrit} is satisfied.
{\hfill $\Box$}

\begin{example}
\label{example:PointProcessDiffusion}
Let $(\xi_n)_{n\ge0}$, $(a_n)_{n\ge0}$ and $(b_n)_{n\ge0}$ be
sequences in $\RR$. Assume that $b_n\neq0$ for $n\ge0$ and assume that $X$
satisfies the one-dimensional stochastic differential equation
\begin{align}
  \df{X}_t = a_{N_t}+b_{N_t}X_t\df{t}+\sigma\df{W}_t + (\xi_{N_t}-X_{t-})\df{N}_t,
\end{align}
with initial value $\xi_0$ and $\sigma>0$, where $W$ is an $(\FFF_t)$ Brownian motion independent
of $N$. Put $\mu_t = |X_{t-}|$. Assume that there are $\alpha,\beta>0$
such that
\begin{align}
  |\xi_n| &\le \alpha+\beta n \\
  |a_n / b_n| &\le \alpha + \beta n \\
  |b_n| &\le \alpha.
\end{align}
Then, the second criteria of Corollary \ref{coro:mainPoissonMGCrit} may be applied to obtain that
$\EEE(H\cdot M)$ is a martingale.
\hfill$\circ$
\end{example}

The interpretation of this example is similar to Example
\ref{example:DiffusionPointProcessCoef} with the intensity under the
$Q$-measure evolving as the absolute value of a linear diffusion
process with constant coefficients between jumps. The intensity is,
however, in this example reset to the level $\xi_n$ at the $n$'th jump
of $N$.

\textbf{\textit{Proof of Example \ref{example:PointProcessDiffusion}:\quad}} We want to show that
the second moment condition of Corollary \ref{coro:mainPoissonMGCrit}
is applicable. To this end, we first construct an explicit solution to the stochastic differential equation
defining $X$. With $T_n$ denoting the $n$'th event time for $N$, define the process $W^n$
by $W^n_t = W_{T_n+t}-W_{T_n}$ and define $\FFF^n_t =
\FFF_{T_n+t}$. By Theorem I.12.1 of \cite{MR1796539}, $W^n$ is independent
of $\FFF_{T_n}$ and has the distribution of a Brownian motion. Again
using Theorem I.12.1 of \cite{MR1796539} with the stopping time $T_n+s$, we
have for $0\le s\le t$ that
\begin{align}
  E(W^n_t|\FFF^n_s)
  &=E(W_{T_n+t}-W_{T_n}|\FFF_{T_n+s})\notag\\
  &=E(W_{T_n+t}-W_{T_n+s}|\FFF_{T_n+s})+W_{T_n+s}-W_{T_n}\notag\\
  &=W_{T_n+s}-W_{T_n}
  =W^n_s,
\end{align}
and L{\'e}vy's characterisation Theorem for Brownian motion relative to a
filtration, see \cite{MR1780932}, Theorem IV.33.1, shows that $W^n$ is an
$(\FFF^n_t)$ Brownian motion. We may then use the It{\^o} existence and uniqueness theorem, see Theorem 11.2 of
\cite{MR1780932}, concluding that on the same probability space that carries the
Poisson process $N$, the Brownian motion $W$ and in particular the
$(\FFF^n_t)$ Brownian motion $W^n$, there exist unique processes $X^n$ satisfying the stochastic differential equations
$\df{X}^n_t=a_n+b_nX^n_t\df{t}+\sigma\df{W}^n_t$ with constant initial values
$\xi_n$. Whenever $T_n\le t<T_{n+1}$, we then have
\begin{align}
  X^n_{t-T_n}&=\xi_n+\int_0^{t-T_n} a_n+b_nX^n_s\df{s}
                   +\int_0^{t-T_n} \sigma\df{W^n}_s\notag\\
             &=\xi_n+\int_{T_n}^t a_n+b_nX^n_{s-T_n}\df{s}
                   +\int_{T_n}^t \sigma\df{W}_s.
\end{align}
The process $\sum_{n=0}^\infty
X^n_{t-T_n}1_{[T_n,T_{n+1})}(t)$ thus satisfies the same stochastic differential equation as $X$. By
pathwise uniqueness for each $X^n$, $X_t=\sum_{n=0}^\infty X^n_{t-T_n}1_{[T_n,T_{n+1})}(t)$.

The above deliberations yield an explicit representation for the
stochastic differential equation defining the intensity. Next, we check
that the second moment condition of Corollary \ref{coro:mainPoissonMGCrit} is applicable. With
$S_k=T_k-T_{k-1}$ denoting the sequence of interarrival times, we then
obtain for the moment condition to be investigated that
\begin{align}
  E\exp\left(\int_u^t \log_+|X_{s-}|\df{N_s}\right)
  &\le E\exp\left(\int_u^t \log(1+|X_{s-}|)\df{N_s}\right)\\
  &=E\prod_{k=N_u+1}^{N_t}(1+|X^{k-1}_{T_k-T_{k-1}}|)
  =E\prod_{k=N_u+1}^{N_t}(1+|X^{k-1}_{S_k}|).\notag
\end{align}
In order to obtain the finiteness of this expression, we wish to
condition on $N$. Given a counting process trajectory $\eta$, we refer
to the event times of $\eta$ by $(\tau_n)$, $\tau_0=0$, and we let $(s_n)$ be the
corresponding interarrival times, $s_n=\tau_n-\tau_{n-1}$. We then have
\begin{align}
  E\prod_{k=N_u+1}^{N_t}(1+|X^{k-1}_{S_k}|)
  &=\int E\left(\left.\prod_{k=\eta_u+1}^{\eta_t}(1+|X^{k-1}_{s_k}|)\right|N=\eta\right)\df{N(P)}(\eta).
\end{align}
Next, we argue that given $N$, the variables $(X^{k-1}_{s_k})_{k\ge1}$ are mutually
independent, in the sense that it $N(P)$ almost surely holds that the
conditional distribution of the variables $(X^{k-1}_{s_k})_{k\ge1}$
given $N=\eta$ is the product measure of each of the marginal
conditional distributions.

Applying Theorem V.10.4 of \cite{MR1780932} and the Doob-Dynkin Lemma, see the first
lemma of Section A.IV.3 of \cite{MR1814344}, there is a measurable
mapping $G_{k-1}:C[0,s_k]\to\RR$ such that $X^{k-1}_{s_k}$ is the
transformation under $G_{k-1}$ of the first $s_k$ coordinates of
$W^{k-1}$. We apply this result to obtain the conditional independence of
$X^{k-1}_{s_k}$ given $N=\eta$. As $X^{k-1}_{s_k}$ is a transformation
of $(W^{k-1})^{s_k}$, it will suffice to show that the processes $(W^{k-1})^{s_k}$ are conditionally
independent given $N=\eta$. To this end, we recall
that $W$ is independent of $N$, and note that $(W^{k-1})^{s_k}_t=W_{(\tau_{k-1}+t)\land
  \tau_k}-W_{\tau_{k-1}}$. Therefore, $(W^{k-1})^{s_k}$ is
$\FFF_{\tau_k}$ measurable. By Theorem I.12.1 of \cite{MR1796539}, $W^{k-1}$ is independent of $\FFF_{\tau_{k-1}}$. Inductively, it follows that conditionally on
$N=\eta$, the sequence of processes $(W^{k-1})^{s_k}$ are
mutually independent. Therefore, conditionally on $N$, the variables
$(X^{k-1}_{s_k})_{k\ge1}$ are mutually independent.

Applying this conditional independence, we may now conclude
\begin{align}
  E\prod_{k=N_u+1}^{N_t}(1+|X^{k-1}_{S_k}|)
  &=\int E\left(\left.\prod_{k=\eta_u+1}^{\eta_t}(1+|X^{k-1}_{s_k}|)\right|N=\eta\right)\df{N(P)}(\eta)\notag\\
  &=E\prod_{k=N_u+1}^{N_t}E(1+|X^{k-1}_{S_k}||N).
\end{align}
Next, we develop a simple bound for $E(|X^{k-1}_{S_k}||N)$. Consider
again a counting process path $\eta$, we then almost surely have
$E(|X^{k-1}_{S_k}||N=\eta)=E|X^{k-1}_{s_k}|$, where $X^{k-1}_{s_k}$
is given by $X^{k-1}_{s_k}=\xi_{k-1}+\int_0^{s_k}a_{k-1}+b_{k-1}X^{k-1}_t\df{t}+\sigma
W^{k-1}_{s_k}$. By (3.42) of \cite{MR1999614}, we then find that $X^{k-1}_{s_k}$ is normally distributed
with mean and variance given by
\begin{align}
  EX^{k-1}_{s_k}&=-\frac{a_{k-1}}{b_{k-1}}+\exp(s_kb_{k-1})\left(\xi_{k-1}+\frac{a_{k-1}}{b_{k-1}}\right).\\
  VX^{k-1}_{s_k}&=\sigma^2\int_0^{s_k}\exp(2b_{k-1}(s_k-u))\df{u}.
\end{align}
By our assumptions on $a_k$, $b_k$ and $\xi_k$, we then obtain
\begin{align}
  E|X^{k-1}_{s_k}|
  &\le |EX^{k-1}_{s_k}|
  +\sqrt{VX^{k-1}_{s_k}}E(X^{k-1}_{s_k}-EX^{k-1}_{s_k})/\sqrt{VX^{k-1}_{s_k}}\\
  &\le\left|\frac{a_{k-1}}{b_{k-1}}\right|+\exp(s_kb_{k-1})\left(|\xi_{k-1}|+\left|\frac{a_{k-1}}{b_{k-1}}\right|\right)+\sqrt{2/\pi}\sigma\sqrt{s_k}\exp(2s_k b_{k-1})\notag\\
  &\le\alpha+\beta (k-1)+2\exp(s_k\alpha)(\alpha+\beta (k-1))+\sqrt{2/\pi}\sigma\sqrt{s_k}\exp(2s_k\alpha).\notag
\end{align}
Therefore, we see that by defining $\alpha^*(v)=\alpha+2\alpha\exp(v\alpha)+\sqrt{2/\pi}\sigma\sqrt{v}\exp(2v\alpha)$
and $\beta^*(v)=\beta+2\beta\exp(v\alpha)$, we have
$E|X^{k-1}_{s_k}|\le\alpha^*(s_k)+\beta^*(s_k)(k-1)$. Next, note that
for $k\le N_t$, it holds that $T_k\le T_{N_t}\le t$. Therefore, for
any $k$ with $N_u+1\le k\le N_t$, it holds that $S_k\le t$. As $\alpha^*$ and
$\beta^*$ are increasing, we then find
\begin{align}
  E\prod_{k=N_u+1}^{N_t}E(1+|X^{k-1}_{S_k}||N)
  &\le E\prod_{k=N_u+1}^{N_t}\left(1+\alpha^*(S_k)+\beta^*(S_k)(k-1)\right)\notag\\
  &\le E\prod_{k=N_u+1}^{N_t}\left(1+\alpha^*(t)+\beta^*(t)(k-1)\right)\notag\\
  &= E\exp\left(\int_u^t \log(1+\alpha^*(t)+\beta^*(t)N_{s-})\df{N_s}\right).
\end{align}
Proceeding as in the the proof of Example
\ref{example:ExistenceLinearChange} using the second moment condition
of Corollary \ref{coro:mainPoissonMGCrit}, it follows that for
$\eps>0$ small enough and $0\le u\le t$ with $t-u\le\eps$, the above
is finite, and so the moment condition is satisfied.
{\hfill $\Box$}

Examples \ref{example:DiffusionPointProcessCoef} and \ref{example:PointProcessDiffusion} show how Corollary
\ref{coro:mainPoissonMGCrit} may be used to construct counting
processes with intensities not adapted to the filtration induced by
$N$ itself. Note that by Corollary 11.5.3 of \cite{MR2057928}, $W$ is always independent of
$N$, so the independence requirements in the above are mentioned only for
clarity. Also note that in Example
\ref{example:DiffusionPointProcessCoef}, the required bounds on the
coefficients hold independently of the norms on $\NN_0^d$, $\RR^d$ and $\MM(d,d)$
chosen, since all norms on finite-dimensional vector spaces are
equivalent.

Our next result, Example \ref{example:Hawkes}, yields a change of measure to a probability measure where the counting process is a
multidimensional Hawkes process. Such results are of interest in the context of models
such as those considered in e.g. \cite{CarstensenEtAl2010}. In general, many specifications of
$\phi$ and $h$ will yield exploding counting processes and there will
exist no measure change yielding the required intensity change.

\begin{example}
\label{example:Hawkes}
Consider mappings $\phi_i:\RR\to[0,\infty)$ and
$h_{ij}:[0,\infty)\to\RR$. Define
\begin{align}
  \mu^i_t=\phi_i\left(\sum_{j=1}^d\int_0^{t-} h_{ij}(t-s)\df{N}^j_s\right).
\end{align}
If $\phi_i$ is Borel measurable with $\phi_i(x)\le |x|$ and $h_{ij}$ is bounded, then $\EEE(H\cdot M)$ is a martingale.
\hfill$\circ$
\end{example}

\textbf{\textit{Proof of Example \ref{example:Hawkes}:\quad}} By Lemma
\ref{lemma:PredictableHawkesIntensity}, the process
$\sum_{j=1}^d\int_0^{t-}h_{ij}(t-s)\df{N}^j_s$ is predictable. As
$\phi_i$ is Borel measurable, it then follows that $\mu^i$ is predictable. As $\phi_i$ is
nonnegative, $\mu$ is nonnegative. And by stopping at event times, we find
that $\mu$ is locally bounded. Thus, $\mu$ is nonnegative, predictable and
locally bounded. Letting $c>0$ be such that
$\|h_{ij}\|_\infty\le c$ for all $i,j\le d$, we obtain
\begin{align}
  \mu^i_t \le  \left|\sum_{j=1}^d\int_0^{t-} h_{ij}(t-s)\df{N}^j_s\right|
  \le \sum_{j=1}^d\int_0^{t-} |h_{ij}(t-s)|\df{N}^j_s
  \le c \sum_{j=1}^d N^j_{t-},
\end{align}
and the result follows from Example \ref{example:ExistenceLinearChange}.
{\hfill $\Box$}

The above examples all give various types of sufficient criteria for
the martingale property of $\EEE(H\cdot M)$ using Corollary
\ref{coro:mainPoissonMGCrit}. As an aside, we may ask whether the
classical necessary and sufficient criterion for nonexplosion for
piecewise constant intensities, see Theorem 2.3.2 of \cite{MR1600720},
may be replicated as a criterion for the martingale property of
$\EEE(H\cdot M)$. The following example shows that this is the case.

\begin{example}
\label{example:Piecewise}
Let $d=1$, let $(\alpha_n)$ be a sequence of positive numbers and let
$\mu_t=\alpha_{N_{t-}}$. Then $\EEE(H\cdot M)$ is a
martingale if and only if $\sum_{n=0}^\infty\frac{1}{\alpha_n}$ is divergent.
\hfill$\circ$
\end{example}

\textbf{\textit{Proof of Example \ref{example:Piecewise}:\quad}} Let $T_n$ be the $n$'th jump time of $N$, then $(T_n)$ is a localising
sequence. We have
\begin{align}
  E\EEE(\mu \cdot M-M)_{T_n}
  &=E\exp\left(T_n-\int_0^{T_n}\mu_s\df{s}+\int_0^{T_n}\log \mu_s\df{N}_s\right)\notag\\
  &=E\exp\left(-\sum_{k=1}^n(\alpha_{k-1}-1)(T_n-T_{n-1})+\sum_{k=1}^n\log\alpha_{k-1}\right)\notag\\
  &=\prod_{k=1}^n\alpha_{k-1}(1-(1-\alpha_{k-1}))^{-1}
  =1,
\end{align}
so $\EEE(M)^{T_n}$ is a uniformly integrable martingale by Lemma \ref{lemma:SimpleEMCrit}. Therefore, by Lemma
\ref{lemma:LocalLimitEMCrit}, $\EEE(\mu\cdot M-M)$ is a martingale
if and only if $\lim_n E\EEE(\mu\cdot M-M)_{T_n}1_{(T_n\le t)}$ is zero for all
$t\ge0$. Now let $(\Omega',\FFF',P')$ be an auxiliary probability
space endowed with a sequence $(U_n)$ of independent exponentially
distributed variables, where $U_n$ has intensity $\alpha_n$. Let $P_n$
be the measure with Radon-Nikodym derivative $\EEE(\mu\cdot M-M)_{T_n}$ with
respect to $P$. By Lemma \ref{lemma:PointProcessMeasChg}, under $P_n$,
$N$ has intensity $\mu1_{[0,T_n]}+1_{(T_n,\infty)}$. In particular, the distribution
of $T_n$ under $P_n$ is then the same as the distribution of
$\sum_{k=1}^nU_k$ under $P'$, and so
\begin{align}
  \lim_n E\EEE(M)_{T_n}1_{(T_n\le t)}
  &=\lim_n P_n(T_n\le t)\notag\\
  &=\lim_n P'\left(\sum_{k=1}^n U_k\le t\right)
  =P'\left(\sum_{k=1}^\infty U_k\le t\right),
\end{align}
since $\cap_{n=1}^\infty (\sum_{k=1}^n U_k\le t)=(\sum_{k=1}^\infty
U_k\le t)$. Now, as $\sum_{k=1}^\infty \frac{1}{\alpha_k}$ diverges, Theorem
2.3.2 of \cite{MR1600720} shows that $\sum_{k=1}^\infty U_k$ is almost
surely infinite, so $P'(\sum_{k=1}^\infty U_k\le t)=0$. The result now follows from Lemma \ref{lemma:LocalLimitEMCrit}.
{\hfill $\Box$}

\section{Proofs of the main results}
\label{sec:SuffUI}

In this section, we present the proofs of the results stated in
Section \ref{sec:Summary}. We begin by recalling two 
folklore results on supermartingales and exponential martingales.

\begin{lemma}
\label{lemma:SimpleSuperMGCrit}
Let $X$ be a nonnegative supermartingale. Then $X$ is a uniformly
integrable martingale if and only if $EX_\infty=EX_0$, and $X$ is a
martingale if and only if it holds for all $t\ge0$ that $EX_t=EX_0$.
\end{lemma}

Recall that if $M$ is a local martingale with $\Delta M\ge-1$ and
initial value zero, $\EEE(M)$ is a nonnegative local martingale and a
supermartingale, $E\EEE(M)_t\le 1$ and $\EEE(M)_\infty$ always exists
as an almost sure limit with $E\EEE(M)_\infty\le 1$. Applying Lemma
\ref{lemma:SimpleSuperMGCrit} to the case of Dol{\'e}ans-Dade
exponentials then yields the following useful result.

\begin{lemma}
\label{lemma:SimpleEMCrit}
Let $M$ be a local martingale with $\Delta M\ge-1$ and initial value zero. $\EEE(M)$ is a uniformly integrable martingale if and only if $E\EEE(M)_\infty=1$, and
$\EEE(M)$ is a martingale if and only if $E\EEE(M)_t=1$ for all $t\ge0$.
\end{lemma}

Now consider given a $d$-dimensional nonexplosive counting process $N$ with
nonnegative, predictable and locally bounded intensity $\lambda$ as well as another
nonnegative, predictable and locally bounded process $\mu$ which is
$\lambda$-compatible. As in Section \ref{sec:Summary}, $M$ is the $d$-dimensional local martingale defined by $M^i_t = N^i_t -
\int_0^t\lambda^i_s\df{s}$. Furthermore, we also use the notation that $\gamma^i=\mu^i_t(\lambda^i_t)^{-1}$
and $H^i_t=\gamma_t^i-1$. Recall that the assumption that $\mu$ is
$\lambda$-compatible by convention implies that both $\gamma$ and $H$ are locally bounded. Integrals are vector integrals in the sense that
$H\cdot M$ denotes the one-dimensional process defined by $H\cdot M =
\sum_{i=1}^d H^i\cdot M^i$.

We first prove Lemma \ref{lemma:PointProcessMeasChg}, the result
stated in Section \ref{sec:Summary} as the reason for taking interest
in the martingale property of $\EEE(H\cdot M)$ when considering
changing the intensity of a counting process. Recall that $\Pi^*_p$
denotes the dual predictable projection, see Definition 5.21 of
\cite{MR1219534}.

\begin{lemma}
\label{lemma:GirsanovMart}
Let $M$ be a local martingale with $\Delta M\ge-1$ and let $T$ be a
stopping time. Assume that $\EEE(M)^T$ is a uniformly integrable
martingale. Let $Q$ be the probability measure having Radon-Nikodym derivative $\EEE(M)_T$ with
respect to $P$. If $L$ is a local martingale under $P$ such that $[L,M^T]$ is locally integrable under $P$, then
$L-\langle L,M^T\rangle$ is a local martingale under $Q$, where the
angle bracket is calculated under $P$.
\end{lemma}

\begin{proof}
Assume given a process $L$ which is a local martingale under $P$ such
that $[L,M^T]$ is locally integrable under $P$. We then find that Theorem
III.41 of \cite{MR2273672} applies and yields that the process given by $L_u -
\int_0^u\EEE(M^T)_{s-}^{-1}\df{\langle \EEE(M^T),L\rangle_s$ is a $Q$ local martingale, where the angle bracket is calculated under
  $P$. Noting that
\begin{align}
  L_u - \int_0^u \frac{1}{\EEE(M^T)_{s-}}\df{\langle \EEE(M^T),L\rangle}_s}
  &=L_u -\int_0^u \frac{1}{\EEE(M^T)_{s-}}\df{\langle \EEE(M^T)_-\cdot M^t,L\rangle}_s\notag\\
  &=L_u - \langle L,M^T\rangle_u,
\end{align}
the result follows.
\end{proof}

\textbf{\textit{Proof of Lemma \ref{lemma:PointProcessMeasChg}:\quad}}
Fix a stopping time $T$. By definition, $Q$ has Radon-Nikodym derivative $\EEE(H\cdot M)_T$ with
respect to $P$. We wish to apply Lemma \ref{lemma:GirsanovMart}. We first check that $[M^i,(H\cdot M)^T]$ is
locally integrable under $P$. Note that $[M^i,M^j]_t = \sum_{0<s\le t}\Delta M^i_s\Delta
M^j_s=\sum_{0<s\le t}\Delta N^i_s\Delta N^j_s=[N^i,N^j]$, since $M^i$ has finite
variation,  in particular $[M^i]=N^i$. As the coordinates of $N$ have
no common jumps, we have $[M^i,(H\cdot M)^T]=H^i1_{[0,T]}\cdot
[N^i]$. Because we have assumed that $H$ is locally bounded, this is locally
integrable. From Lemma \ref{lemma:GirsanovMart}, we then conclude that $M^i- \langle M^i,(H\cdot
M)^T\rangle$ is a local martingale under $Q$. Next, under $P$,
we have $(\Pi^*_pN^i)_t = \int_0^t \lambda^i_s\df{s}$, and $H$ and $1_{[0,T]}$
are predictable. Therefore, we obtain $\langle M^i,(H\cdot M)^T\rangle_s =
\Pi^*_p(H^i1_{[0,T]}\cdot [N^i])_s=\int_0^s H^i_u1_{(u\le
  T)}\lambda_u^i\df{u}$, which allows us to conclude that
\begin{align}
  M^i_s-\langle M^i, (H\cdot M)^T\rangle_s
  &=N^i_s-\int_0^s\lambda^i_s\df{s}-\int_0^s H^i_u1_{(u\le T)}\lambda^i_u \df{u}\notag\\
  &=N^i_t-\int_0^s \mu^i_u1_{[0,T]}(u)+\lambda^i_u1_{(T,\infty)}(u)\df{u}.
\end{align}
This proves that under $Q$, $N$ has intensity
$1_{[0,T]}\mu+1_{(T,\infty)}\lambda$. The results for the case where
$\EEE(H\cdot M)$ is a martingale then follows by considering stopping
times which are constant.
{\hfill $\Box$}

Next, we prove Lemma \ref{lemma:QPEquivalence}, which yields a
sufficient criterion for the probability measure $Q$ constructed using an exponential
martingale to be equivalent to our starting probability measure $P$. For this
purpose, we first require the following lemma.

\begin{lemma}
\label{lemma:LebesgueNullZeroes}
Let $N$ have intensity $\lambda$. If $X$ is a process which is nonnegative,
predictable and locally bounded, and it holds almost surely that pathwisely, the set
of zeroes of $X$ has Lebesgue measure zero, then it almost surely
holds that the zeroes of $X$ are disjoint from the jump times of
$N^i$ for all $i$.
\end{lemma}

\begin{proof}
As $X$ is predictable, the set of zeroes of $X$ is a
predictable set. Thus, the integral process $\int_0^t 1_{(X_s=0)}\df{M^i}_s$ is a local
martingale. Let $(T_n)$ be a localising sequence such that $\int_0^t
1_{(X_s=0)}1_{(t\le T_n)}\df{N^i}_s$ is bounded and $\int_0^t
1_{(X_s=0)}1_{(t\le T_n)}\df{M^i}_s$ is a true martingale. Then $E\int_0^t
1_{(X_s=0)}1_{(t\le T_n)}\df{M^i}_s=0$, and so by our assumptions,
$E\int_0^t 1_{(X_s=0)}1_{(t\le T_n)}\df{N^i_s}=0$ as well, leading us
to conclude that $\int_0^\infty 1_{(X_s=0)}\df{N^i_s}$ is
almost surely zero. This implies that almost surely, the set of zeroes
of $X$ is disjoint from the jump times of $N^i$. As the coordinate
$i$ was arbitrary, the result follows.
\end{proof}

\textbf{\textit{Proof of Lemma \ref{lemma:QPEquivalence}:\quad}}
Note that $\Delta(H\cdot M)_t = \sum_{i=1}^d H^i_t\Delta N^i_t$. By
Lemma \ref{lemma:LebesgueNullZeroes}, the set of zeroes of $\mu^i$ is
disjoint from the jump times of $N^i$. Therefore, the set of zeroes of
$\gamma^i$ is disjoint from the jump times of $N^i$ as well, and so
the set where $H^i$ is $-1$ is disjoint from the jump times of
$N^i$. We conclude that almost surely, $H\cdot M$ has no jumps of size
$-1$. Theorem I.4.61 of \cite{MR1943877} then shows that $\EEE(H\cdot M)$ is
almost surely positive.
{\hfill $\Box$}

Finally, we prove Theorem \ref{theorem:mainPointProcessMGCrit} and its
corollary. We first state the two main theorems of \cite{MR489492} which we will
apply to integrals of compensated counting processes in order to
obtain our results. The two main theorems from that article are Theorem
III.1 and Theorem III.7, given below.

\begin{theorem}
\label{theorem:LepingleMemin}
Let $M$ be a local martingale with initial value zero and jumps satisfying $\Delta
M\ge -1$. Let $R=\inf\{t\ge0\mid \Delta M_t = -1\}$. Define $B$ by putting $B_t = \frac{1}{2}[M^c]_{t\land
  R}+\sum_{0<s\le t\land R}(1+\Delta
M_s)\log(1+\Delta M_s)-\Delta M_s$. If $B$ is locally integrable and $\exp(\Pi^*_pB_\infty)$ is
integrable, then $\EEE(M)$ is a uniformly integrable martingale.
\end{theorem}

\begin{theorem}
\label{theorem:LepingleMemin2}
Let $M$ be a local martingale with initial value zero and $\Delta
M>-1$. Define $A$ by putting $A_t=\frac{1}{2}[M^c]_t+\sum_{0<s\le t}\log(1+\Delta M_s)-\frac{\Delta
  M_s}{1+\Delta M_s}$. If $\exp(A_\infty)$ is integrable, then $\EEE(M)$ is
a uniformly integrable martingale.
\end{theorem}

The following two lemmas are ingredients for the proof of Theorem
\ref{theorem:mainPointProcessMGCrit}. The first lemma allows us to
restrict our attention to small deterministic time intervals when proving the martingale property of
exponential martingales. This technique is well-known, see for example
Corollary 3.5.14 of \cite{MR917065}, and so we omit the proof. The second lemma decomposes
an exponential martingale into the product of two exponential
martingales, corresponding to successive changes of intensity from
$\lambda$ to $\mu$ and $\mu$ to $\mu+\nu$. This will, colloquially speaking, allow us to consider
the large and small parts of $\mu$ separately when proving the martingale property.

\begin{lemma}
\label{lemma:epsMGCrit}
Let $M$ be a local martingale with $\Delta M\ge-1$, and let
$\eps>0$. If $\EEE(M^t-M^u)$ is a martingale whenever $0\le u\le
t$ with $t-u\le\eps$, where $M^t$ denotes the process $M$ stopped at
time $t$, then $\EEE(M)$ is a martingale.
\end{lemma}

\begin{lemma}
\label{lemma:SuccessiveIntensityChange}
Let $\nu$ be nonnegative, predictable and locally bounded. Assume that
$\mu$ is $\lambda$-compatible and that $\mu+\nu$ is $\mu$-compatible. Then $\mu+\nu$ is also
$\lambda$-compatible. Define three processes $(H_\lambda^{\mu+\nu})^i_t=(\mu^i_t+\nu^i_t)(\lambda_t^i)^{-1}-1$,
$(H_\lambda^\mu)^i_t=\mu^i_t(\lambda^i_t)^{-1}-1$ and
$(H_\mu^{\mu+\nu})^i_t=(\mu^i_t+\nu^i_t)(\mu^i_t)^{-1}-1$. Define
$d$-dimensional processes $M^\lambda$ and $M^\mu$ by putting $(M^\lambda)^i_t =
N^i_t-\int_0^t \lambda^i_s\df{s}$ and $(M^\mu)^i=N^i_t-\int_0^t
\mu^i_s\df{s}$. It then holds that $\EEE(H_\lambda^{\mu+\nu}\cdot M^\lambda)=\EEE(H_\lambda^\mu\cdot
M^\lambda)\EEE(H_\mu^{\mu+\nu}\cdot M^\mu)$.
\end{lemma}

\begin{proof} That $\mu+\nu$ is
$\lambda$-compatible follows as $\mu+\nu$ is $\mu$-compatible and
$\mu$ is $\lambda$-compatible. Furthermore, $M^\lambda$ and $M^\mu$
are processes of finite variation, so we find $[H_\lambda^\mu\cdot
M^\lambda,H_\mu^{\mu+\nu}\cdot M^\mu]_t=\sum_{i=1}^d\int_0^t
(H_\lambda^\mu)^i_s (H_\mu^{\mu+\nu})^i_s\df{N^i}_s$. Therefore, we obtain $\EEE(H_\lambda^\mu\cdot M^\lambda)\EEE(H_\mu^{\mu+\nu}\cdot M^\mu)
  =\EEE(H_\lambda^\mu\cdot M^\lambda+H_\mu^{\mu+\nu}\cdot
  M^\mu+H_\lambda^\mu H_\mu^{\mu+\nu}\cdot N)$ by Theorem II.38 of
  \cite{MR2273672}. By elementary calculations,
\begin{align}
  H_\lambda^\mu\cdot M^\lambda+H_\mu^{\mu+\nu}\cdot
  M^\mu+H_\lambda^\mu H_\mu^{\mu+\nu}\cdot N
  &=H^{\mu+\nu}_\lambda \cdot M^\lambda,
\end{align}
and so the result holds.
\end{proof}

\textbf{\textit{Proof of Theorem \ref{theorem:mainPointProcessMGCrit}:\quad}}
By Lemma \ref{lemma:epsMGCrit}, it suffices to show the martingale
property of $\EEE((H\cdot M)^t-(H\cdot M)^u)$ when $0\le u\le t$ with
$t-u\le\eps$. Let such a pair of $u$ and $t$ be
given and let $L=(H\cdot M)^t-(H\cdot M)^u$. With $R$ and $B$ as in
Theorem \ref{theorem:LepingleMemin}, we have for $r\ge0$ that
\begin{align}
 B_r &=\sum_{i=1}^d\int_0^r 1_{[0,R]}(s)1_{[u,t]}(s)((1+H^i_s)\log(1+H^i_s)-H^i_s)\df{N^i}_s.
\end{align}
From this, we obtain that $B$ is locally integrable, and as $1_{[0,R]}$ is a predictable process, we have
\begin{align}
  (\Pi^*_pB)_\infty
  &=\sum_{i=1}^d\int_u^t 1_{[0,R]}(s)(\gamma^i_s\log\gamma^i_s-(\gamma^i_s-1))\lambda^i_s\df{s}\notag\\
  &\le \sum_{i=1}^d\int_u^t (\gamma^i_s\log\gamma^i_s-(\gamma^i_s-1))\lambda^i_s\df{s}.
\end{align}
Therefore, if the first integrability criterion is satisfied, $\EEE(L)$ is a uniformly integrable martingale by Theorem
\ref{theorem:LepingleMemin}, in particular a martingale. This proves
the first claim.

Next, we consider the case where the second integrability criterion is satisfied. We will use Lemma
\ref{lemma:SuccessiveIntensityChange} to prove that $\EEE((H\cdot
M)^t-(H\cdot M)^u)$ is a martingale in this case. To this end, we define predictable
$d$-dimensional processes $\mu^-$ and $\mu^+$ by
\begin{align}
  (\mu^-)^i_s &= \mu^i_s1_{(\mu^i_s\le\lambda^i_s)}+\lambda^i_s1_{(\mu^i_s>\lambda^i_s)}\\
  (\mu^+)^i_s &= (\mu^i_s-\lambda^i_s)1_{(\mu^i_s> \lambda^i_s)}.
\end{align}
We then have $\mu=\mu^++\mu^-$. Also define two processes
$(\gamma^*)^i=(\mu^-)^i(\lambda^i)^{-1}$ and
$(\gamma^{**})^i=\mu^i ((\mu^-)^i)^{-1}$, and $H^*=\gamma^*-1$ and
$H^{**}=\gamma^{**}-1$. Now, as $\lambda$ and $\mu$ are predictable,
$\mu^-$ and $\mu^+$ are predictable as well. Furthermore, $\mu^-$
and $\mu^+$ are both nonnegative and locally bounded. By inspection,
$\mu^-$ is $\lambda$-compatible and $\mu$ is $\mu^-$-compatible. Now define $(M^-)^i_t = N^i_t - \int_0^t (1_{(u,t]}(s)(H^*)^i_s+1)\lambda^i_s \df{s}$. Put $L^*=(H^*\cdot M)^t-(H^*\cdot M)^u$ and $L^{**}=(H^{**}\cdot
M^-)^t-(H^{**}\cdot M^-)^u$. Note that $L^*=H^*1_{(u,t]}\cdot
M$, $L^{**}=H^{**}1_{(u,t]}\cdot M^-$ and $L=H1_{(u,t]}\cdot M$. Invoking Lemma
\ref{lemma:SuccessiveIntensityChange}, we obtain
$\EEE(L)=\EEE(L^*)\EEE(L^{**})$. We will apply Theorem \ref{theorem:LepingleMemin} to the local
martingale $L^*$. By the same calculations
as earlier, noting that $(1+x)\log(1+x)\le 0$ when $-1\le x\le 0$, we obtain
\begin{align}
   E\exp\left(\sum_{i=1}^d\int_u^t \lambda^i_s\df{s}\right)
  &\le E\exp\left(\sum_{i=1}^d \int_u^t \lambda^i_s\df{s}+\int_u^t \log_+\gamma^i_s \df{N^i}_s\right)<\infty,
\end{align}
so Theorem \ref{theorem:LepingleMemin} shows that $\EEE(L^*)$ is
a uniformly integrable martingale. Let $Q$ be the measure with
Radon-Nikodym derivative $\EEE(L^*)_\infty$ with respect to $P$. We
then have $E^P\EEE(L)_\infty=E^Q\EEE(L^{**})_\infty$. To show that
$\EEE(L)$ is a uniformly integrable martingale, it suffices to show that this is equal to one. To do so, we will apply Theorem
\ref{theorem:LepingleMemin2} to show that $\EEE(L^{**})$ is a
uniformly integrable martingale under $Q$. To this end, first note
that by Lemma \ref{lemma:PointProcessMeasChg}, $N^i$ has intensity
$(1_{(u,t]}(H^*)^i+1)\lambda^i$ under $Q$. Therefore, $M^-$ is a local martingale
under $Q$, and so $L^{**}$ is a local martingale under $Q$ as well. Next, $(H^{**})^i_t
  =(\gamma^{**})^i_t-1
   =1_{(\mu^i_t\le\lambda^i_t)}+\gamma^i_t1_{(\mu^i_t>\lambda^i_t)}-1\ge0$,
so $\Delta L^{**}\ge0>-1$, and therefore Theorem
\ref{theorem:LepingleMemin2} is applicable. Now, with $A$ as in Theorem \ref{theorem:LepingleMemin2}, we have
\begin{align}
  A_\infty &=\frac{1}{2}[(L^{**})^c]_\infty+\sum_{0<s}\log(1+\Delta L^{**}_s)-\frac{\Delta  L^{**}_s}{1+\Delta L^{**}_s}\notag\\
  &\le\sum_{i=1}^d\int_u^t \log\frac{\mu^i_s}{(\mu^-)^i_s}\df{N^i}_s
 =\sum_{i=1}^d\int_u^t \log_+\gamma^i_s \df{N^i}_s.
\end{align}
Also, since $-1\le H^*\le 0$, we find
$\EEE(L^*)_\infty\le\exp(\sum_{i=1}^d \int_u^t \lambda^i_s\df{s})$, which leads to
\begin{align}
  E^Q\exp(A_\infty)
  &=E\EEE(L^*)_\infty \exp(A_\infty)
  \le E\exp\left(\sum_{i=1}^d \int_u^t \lambda^i_s\df{s}+\int_u^t \log_+\gamma^i_s \df{N^i}_s\right),
\end{align}
which is finite by assumption. Theorem \ref{theorem:LepingleMemin2}
then shows that $L^{**}$ is a uniformly integrable martingale under $Q$, so $E^Q\EEE(L^{**})_\infty=1$, from
which we conclude $E^P\EEE(L)_\infty=1$. Thus, $\EEE(L)$ is a uniformly integrable
martingale, in particular a martingale. This completes the proof.
{\hfill $\Box$}

\textbf{\textit{Proof of Corollary \ref{coro:mainPoissonMGCrit}:\quad}}
First note that
\begin{align}
  x\log x-(x-1)&\le 1+x\log x\le 1+x\log_+x
\end{align}
for $x\ge0$. Therefore, as $\lambda=1$, the first moment condition of Theorem
\ref{theorem:mainPointProcessMGCrit} reduces to the first moment
condition in the statement of the corollary. Furthermore, we have $E\exp(\sum_{i=1}^d
\int_u^t\lambda_s^i\df{s}+\int_u^t\log_+
\gamma^i_s\df{N^i}_s)=e^{d(t-u)}E\exp(\sum_{i=1}^d\int_u^t\log_+
\gamma^i_s\df{N^i}_s)$ as $\lambda=1$, and so the result for the second moment condition of the corollary
follows. This completes the proof.
{\hfill $\Box$}

\appendix
\section{Supplementary results}
\label{sec:Proofs}


\begin{lemma}
\label{lemma:PoissonExpMoment}
Let $Z$ be Poisson distributed with parameter $\mu$. Then
$\exp(\eps Z\log Z)$ is integrable whenever $0\le \eps<1$.
\end{lemma}
\begin{proof}
This follows by an application of Stirling's formula, see (6.11.2) of
\cite{MR2933767}, and comparison with a geometric series.
\end{proof}

\begin{lemma}
\label{lemma:DiffusionPointProcessCoefExistence}
Consider $A:\NN_0^d\times \RR^d_+\to\RR^d$, $B:\NN^d_0\times \RR^d_+\to\MM(d,d)$ and
$\sigma:\NN^d_0\times \RR^d_+\to\MM(d,d)$ such that
$A(\eta,\cdot)$, $B(\eta,\cdot)$ and $\sigma(\eta,\cdot)$ are bounded and
continuous for $\eta\in\NN_0^d$. Let $W$ be a $d$-dimensional $(\FFF_t)$
Brownian motion. Let $T^i_n$ be the $n$'th event time for $N^i$ and
let $Z^i_t = t- T^i_{N^i_t}$. The stochatic differential equation
\begin{align}
\df{X}_t = (A(N_t,Z_t)+B(N_t,Z_t)X_t)\df{t} +
\sigma(N_t,Z_t) \df{W}_t
\end{align}
is exact, in the sense that for any initial value, it has a pathwise
unique solution. Defining $C_t = \exp(-\int_0^t B(N_s,Z_s)\df{s})$, the solution is
\begin{align}
  X_t &= C^{-1}_t\left(X_0+\int_0^t C_sA(N_s,Z_s)\df{s}+\int_0^t C_s\sigma(N_s,Z_s)\df{W}_s\right).
\end{align}
\end{lemma}
\begin{proof}
Let $\tilde{A}_s=A(N_s,Z_s)$, and define $\tilde{B}$ and
$\tilde{\sigma}$ analogously. Note that as $N$ and $Z$ are adapted,
$\tilde{A}$ is adapted as well, since $A(\eta,\cdot)$ is continuous
and so Borel measurable for all $\eta\in\NN_0^d$. As the process also is right-continuous and
locally bounded, all integrals are well-defined, and similarly for
$\tilde{B}$ and $\tilde{\sigma}$. Let $X_0$ be some initial value. Assume
that $X$ is a solution to the stochastic differential equation. Note that
each entry of $C_t$ is differentiable as a function of $t$, and $\frac{\df{}}{\df{t}}C^{ij}_t =
(-\tilde{B}_tC_t)^{ij}$. The integration-by-parts formula yields
\begin{align}
  (C_tX_t)_i &=X^i_0+\sum_{j=1}^d \int_0^t C^{ij}_s\df{X^j}_s -\int_0^t X_s^j(\tilde{B}_sC_s)^{ij}\df{s}.
\end{align}
This implies $(C_tX_t)_i =X_0^i+\sum_{j=1}^d \int_0^t C^{ij}_s\tilde{A}^j_s\df{s}+\int_0^t
C^{ij}_s\sum_{k=1}^d\tilde{\sigma}^{jk}_s\df{W}^k_s$, since $X$ is a
solution, leading to
\begin{align}
  X_t = C^{-1}_t\left(X_0+\int_0^t C_s A(N_s,Z_s)\df{s}+\int_0^t
    C_s \sigma(N_s,Z_s)\df{W}_s\right).
\end{align}
This proves pathwise uniqueness. Applying the integration-by-parts
formula to the above shows that the proposed solution in fact is a
solution, yielding existence.
\end{proof}

\begin{lemma}
\label{lemma:NormalExpPowerMoment}
Let $X$ be a $d$-dimensional normally distributed variable with mean
$\xi$ and positive definite variance $\Sigma$. Let $c>0$ and $0<\eps<1$. Then $\exp(c\|X\|_2^{1+\eps})$
is integrable. Furthermore, defining $a(c,\eps)=2^{1+\eps}c$ and
$b(c,\eps)=16^{(1+\eps)/(1-\eps)}c^{2/(1-\eps)}$, it
holds that
\begin{align}
  E\exp(c\|X\|_2^{1+\eps})
  \le k_d\exp(a(c,\eps)\|\xi\|^{1+\eps})\exp\left(b(c,\eps)\|\Sigma\|_2^\frac{1+\eps}{1-\eps}\right),
\end{align}
where $k_d=A_dm_{d-1}(\sqrt{2}\sqrt{\pi^{d-1}})^{-1}$, $A_d$ is
the area of the unit sphere in $d$ dimensions and $m_d$ is
the $d$'th absolute moment of the standard normal distribution.
\end{lemma}
\begin{proof}
By \cite{MR792300}, p. 181, $\Sigma$ has a unique symmetric positive definite square
root $\Sigma^{1/2}$ such that $\Sigma =
(\Sigma^{1/2})^2$. Furthermore, with $Y=\Sigma^{-1/2}(X-\xi)$, it holds that
$X=\xi+\Sigma^\frac{1}{2}Y$, where $Y$ is $d$-dimensionally standard
normally distributed. With $\|\cdot\|_2$ denoting the operator norm
induced by the Euclidean norm, we get
\begin{align}
  E\exp(c\|X\|_2^{1+\eps})
  &= E\exp(c\|\xi+\Sigma^\frac{1}{2}Y\|_2^{1+\eps})
  \le E\exp(c(\|\xi\|_2+\|\Sigma^\frac{1}{2}Y\|_2)^{1+\eps})\notag\\
  &\le E\exp(c2^{1+\eps}(\|\xi\|^{1+\eps}_2+\|\Sigma^\frac{1}{2}Y\|_2^{1+\eps}))\notag\\
  &\le\exp(c2^{1+\eps}\|\xi\|^{1+\eps})E\exp\left(c2^{1+\eps}\|\Sigma\|_2^{(1+\eps)/2}\|Y\|_2^{1+\eps}\right).
\end{align}
Switching to polar coordinates, we
obtain, with $A_d$ denoting the area of the unit sphere in $d$
dimensions and $C=c2^{1+\eps}\|\Sigma\|_2^{(1+\eps)/2}$,
\begin{align}
  & E\exp\left(c2^{1+\eps}\|\Sigma\|_2^{(1+\eps)/2}\|Y\|_2^{1+\eps}\right)\notag\\
  &=\int_{\RR^d}\exp\left(C\|x\|^{1+\eps}_2\right)\frac{1}{\sqrt{(2\pi)^d}}\exp\left(-\frac{1}{2}\|x\|_2^2\right)\df{x}\notag\\
  &=\frac{A_d}{\sqrt{(2\pi)^{d-1}}}\int_0^\infty\exp\left(Cr^{1+\eps}\right)\frac{1}{\sqrt{2\pi}}\exp\left(-\frac{1}{2}r^2\right)r^{d-1}\df{r}.
\end{align}
Using a change of variables, we obtain the bound
\begin{align}
  &\int_0^\infty\exp\left(Cr^{1+\eps}\right)\frac{1}{\sqrt{2\pi}}\exp\left(-\frac{1}{2}r^2\right)r^{d-1}\df{r}\notag\\
  &\le \int_0^\infty r^{d-1}\frac{1}{\sqrt{2\pi}}\exp\left(-\frac{1}{4}r^2\right)\df{r}\sup_{s\ge0}\exp\left(Cs^{1+\eps}-\frac{1}{4}s^2\right)\notag\\
  &=2^{d/2}\int_0^\infty r^{d-1}\frac{1}{\sqrt{2\pi}}\exp\left(-\frac{1}{2}r^2\right)\df{r}\sup_{s\ge0}\exp\left(Cs^{1+\eps}-\frac{1}{4}s^2\right).
\end{align}
With $m_d$ denoting the $d$'th absolute moment of the standard normal
distribution, we have $\int_0^\infty
r^{d-1}\frac{1}{\sqrt{2\pi}}\exp\left(-\frac{1}{2}r^2\right)\df{r}=\frac{1}{2}m_{d-1}$. Also,
defining $\phi(r)=C r^{1+\eps}-\frac{1}{4}r^2$ for
$r\ge0$, $\phi$ has a global maximum at
$r^*=(2C(1+\eps))^{1/(1-\eps)}$ which satisfies $\phi(r^*)\le
4^{\frac{1+\eps}{1-\eps}}C^{\frac{2}{1-\eps}}$. This allows us to conclude
\begin{align}
  \int_0^\infty\exp\left(Cr^{1+\eps}\right)\frac{1}{\sqrt{2\pi}}\exp\left(-\frac{1}{2}r^2\right)r^{d-1}\df{r}
  &\le 2^{d/2-1}m_{d-1}\exp\left(4^{\frac{1+\eps}{1-\eps}}C^{\frac{2}{1-\eps}}\right).
\end{align}
Recalling our definition of $C$, we have
$4^{\frac{1+\eps}{1-\eps}}C^{\frac{2}{1-\eps}}=16^{\frac{1+\eps}{1-\eps}}c^{\frac{2}{1-\eps}}\|\Sigma\|_2^{\frac{1+\eps}{1-\eps}}$. Therefore,
defining $a(c,\eps)=2^{1+\eps}c$ and
$b(c,\eps)=16^{(1+\eps)/(1-\eps)}c^{2/(1-\eps)}$, we
finally obtain the result.
\end{proof}

\begin{lemma}
\label{lemma:PredictableHawkesIntensity}
Let $N$ be a point process, let $h:\RR_+\to\RR$ be Borel measurable and
define $\mu_t = \int_0^{t-} h(t-s)\df{N}_s$. Then $\mu$ is a
predictable process.
\end{lemma}

\begin{proof}
This follows by monotone convergence and Dynkin class arguments.
\end{proof}

\begin{lemma}
\label{lemma:LocalLimitEMCrit}
Let $(T_n)$ be a localising sequence and assume that $\EEE(M)^{T_n}$
is a martingale. $\EEE(M)$ is a martingale if and
only if $\lim_n E\EEE(M)_{T_n}1_{(T_n\le t)}=0$ for each $t\ge0$.
\end{lemma}
\begin{proof}
By our assumptions on the martingale property of $\EEE(M)^{T_n}$, it
holds that $E\EEE(M)_{T_n}1_{(T_n\le t)}=1-E\EEE(M)_t1_{(T_n>t)}$. By the Dominated Convergence Theorem, $\lim_n
E\EEE(M)_t1_{(T_n>t)}=E\EEE(M)_t$. Thus, $\lim_n E\EEE(M)_{T_n}1_{(T_n\le
  t)}=1-E\EEE(M)_t$, and so Lemma \ref{lemma:SimpleEMCrit} yields
the result.
\end{proof}

\bibliographystyle{unsrt}

\bibliography{full}

\end{document}